\documentclass[12pt]{article}%
\usepackage[latin1]{inputenc}
\usepackage{amsmath}
\usepackage{amsfonts}
\usepackage{amssymb}
\usepackage{graphicx}%
\newtheorem{theorem}{Theorem}

\newtheorem{lemma}[theorem]{Lemma}

\newtheorem{proposition}[theorem]{Proposition}

\newenvironment{proof}[1][Proof]{\noindent\textbf{#1.} }{\ \rule{0.5em}{0.5em}}
\def \R{\mathbb{R} }
\def \Z{\mathbb{Z} }
\def \Q{\mathbb{Q}}
\def \H{\mathbf{H} }

\def \charone{{\pmb 1}}

\def \w{\mathbf{w} }

\def \x{\mathbf{x} }
\def \y{\mathbf{y} }
\def \z{\mathbf{z} }
\def \X{\mathbf{X} }
\def \Y{\mathbf{Y} }
\def \V{\mathbf{V} }
\def \P{\mathbf{P} }
\def \Q{\mathbf{Q} }
\def \U{\mathbf{U} }
\def \SL{\operatorname{SL}}
\def \PSL{\operatorname{PSL}}
\def \one{1\hskip -4 pt 1}
\def \im{\operatorname{Im}}

\begin{document}

\title{Equidistribution of Horocyclic Flows on Complete Hyperbolic Surfaces of Finite Area}
\author{John H. Hubbard and Robyn L. Miller\\Cornell University}
\maketitle

\renewcommand{\today}{September 3, 2007}

\vspace{0.5mm}

\begin{abstract}
We provide a self-contained, accessible introduction to Ratner's Equidistribution Theorem in the special case of horocyclic flow on a complete hyperbolic surface of finite area.  This equidistribution result was first obtained in the early 1980s by Dani and Smillie \cite{DaniSmillie84} and later reappeared as an illustrative special case \cite{Ratner92} of Ratner's work [Rat91-Rat94] on the equidistribution of unipotent flows in homogeneous spaces.  We also prove an interesting probabilistic result due to Breuillard: on the modular surface an arbitrary uncentered random walk on the horocycle through almost any point will fail to equidistribute, even though the horocycles are themselves equidistributed \cite{Breuillard05}.  In many aspects of this exposition we are indebted to Bekka and Mayer's more ambitious survey \cite{Bekka00}, \textit{Ergodic Theory and Topological Dynamics for Group Actions on Homogeneous Spaces}.
\end{abstract}

\vspace{2mm}

\section{Horocycle flow on hyperbolic surfaces}

\begin{figure}[ptb]
\begin{centering}
\includegraphics[scale=0.85]{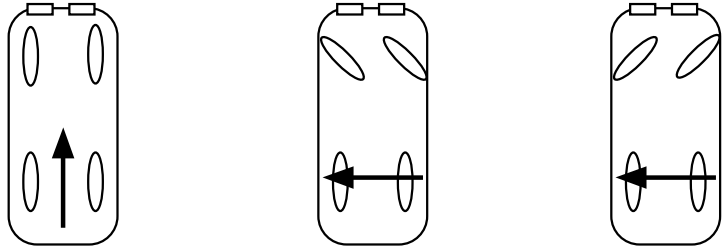}
\caption{\label{carfig} Driving the cars above leads to geodesic flow, positive horocycle flow and
negative horocycle flow
respectively}
\end{centering}
\end{figure}

Let $X$ be a complete hyperbolic surface, perhaps the hyperbolic plane
$H$, and let $\X $ denote the unit tangent bundle $T^{1}(X)$
to $X$ (and $\H=T^1H$). There are three flows on $\X $ which will concern us here. They
are realized by three cars, as represented in Figure \ref{carfig}.

The cars all have their steering wheels locked in position: the first
car drives straight ahead, the second one steers to the left so as to follow a
path of geodesic curvature 1, and the third steers to the right, also
following a path of geodesic curvature 1. All three cars have an arrow painted
on the roof, centered at the rear axle; for the first the arrow points
straight ahead, and for the other two it points sideways -- in the direction
towards which the car is steering for the second car and in the opposite direction for the
third.

The flows at time $t \in\R$ starting at
a point $\x=(x,\xi) \in\X $ are defined as follows:\newline

\begin{enumerate}
\item \textit{The geodesic flow}: put the first car on $X$ with the arrow
pointing in the direction of $\xi$, and drive a distance $t$. The point of
arrival, with the arrow on the car at that point, will be denoted
$\x g(t)$;

\item \textit{The positive horocyclic flow}: put the second car on $X$ with the
arrow pointing in the direction of $\xi$, and drive a distance $t$. The point
of arrival, with the arrow on the car at that point, will be denoted
$\x  u_{+}(t)$;

\item \textit{The negative horocyclic flow}: put the third car on $X$ with the
arrow pointing in the direction of $\xi$, and drive a distance $t$. The point
of arrival, with the arrow on the car at that point, will be denoted
$\x u_{-}(t)$.
\end{enumerate}

We will see when we translate to matrices why it is convenient to
write the flows as \textit{right} actions. \newline\indent The trajectories
followed by these cars are represented in Figure \ref{drivinginplane}.

\begin{figure}[ptb]
\begin{centering}
\includegraphics[scale=0.65]{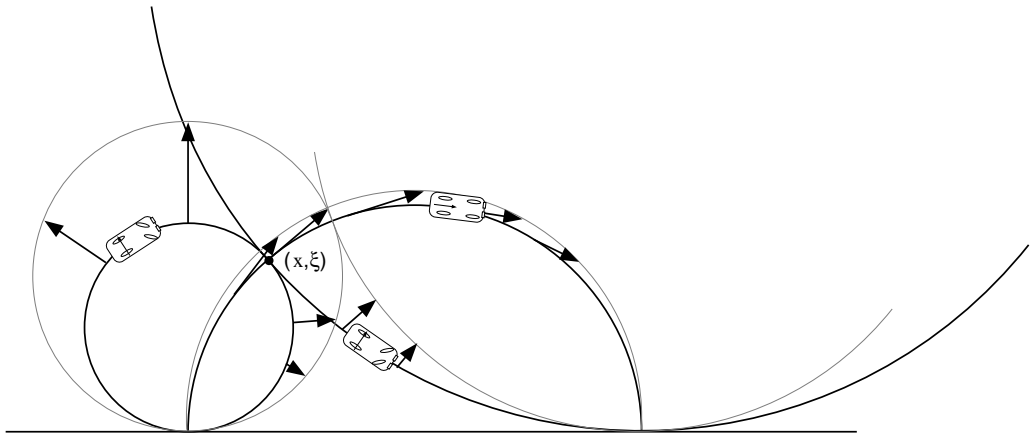}
\caption{\label{drivinginplane} In the upper half-plane model of the hyperbolic plane, the geodesic
passing through $(x,\xi)$ is the semicircle perpendicular to the real axis and tangent at $x$
to $\xi$.  One should remember that it is not a curve in  $\mathbb{H}$, but rather
a curve in $T^1(\mathbb{H})$ and carries its velocity vector with it.  From the
point $(x,\xi)$, the positive horocycle is the circle tangent to the real axis
at the endpoint of the geodesic above and perpendicular to $\xi$ at $x$, whereas
the negative horocycle flow is the circle tangent to the real axis at the origin of the
geodesic, and still perpendicular to $\xi$ at $x$.  We have drawn our
tinkertoys driving along them.}
\end{centering}
\end{figure}

\section{Translation to Matrices}

In less picturesque language (more formal, not more accurate), you can
identify $\X $ with $\Gamma\setminus \PSL_{2}\R$ for some
Fuchsian group $\Gamma$.

\begin{enumerate}
\item The geodesic flow of the point represented by $g\in \PSL_{2}\R$ is
\[
t\mapsto g\left(
\begin{array}
[c]{cc}%
e^{\frac{t}{2}} & 0\\
0 & e^{-\frac{t}{2}}%
\end{array}
\right)  ;
\]

\item The positive horocyclic flow of the point represented by $g\in \PSL_{2}\R$ is
\[
t\mapsto g\left(
\begin{array}
[c]{cc}%
1 & t\\
0 & 1
\end{array}
\right)  ;
\]

\item The negative horocyclic flow of the point represented by $g\in \PSL_{2}\R$ is
\[
t\mapsto g \left(
\begin{array}
[c]{cc}%
1 & 0\\
t & 1
\end{array}
\right)  ;
\]
\end{enumerate}

\indent The standard left action of $\PSL_{2}\R$ on $H$,
which lifts by the derivative to a left action on $\H$ is given by
\begin{equation}
\bmatrix a&b\\c&d\endbmatrix  \cdot z = \frac{az+b}{cz+d} \quad \text{lifting to} \
\bmatrix a&b\\c&d\endbmatrix \cdot(z,\xi) = \left(  \frac{az+b}{cz+d},\frac{\xi}{(cz+d)^{2}}\right)
\end{equation}

\noindent We can then identify $\PSL_2\R$ to $\H$ by
choosing $\mathbf{x_{0}}=(i,i) \in \H$ and setting $\Phi: \PSL_{2}\R \rightarrow \H$ to be
\[
\Phi\left(
\begin{array}
[c]{cc}%
a & b\\
c & d\\
\end{array}
\right)  := \left(
\begin{array}
[c]{cc}%
a & b\\
c & d\\
\end{array}
\right)  \cdot\mathbf{x_{0}} = \left(  \frac{ai+b}{ci+d},\frac{i}{(ci+d)^{2}%
}\right)
\]
\noindent Since
\[
\Phi(\gamma A)=(\gamma A) \cdot\mathbf{x_{0}} = \gamma\cdot(A \cdot
\mathbf{x_{0}})  = \gamma\cdot\Phi(A)
\]
\noindent we see that $\Phi$ induces a diffeomorphism $\Phi_{\Gamma} :
\Gamma\setminus \PSL_{2}\R \rightarrow\Gamma \setminus \X $.

The left action above does \textit{not} induce an action of
$\PSL_2 \R$ on $\Gamma \setminus\X $, but there is an action on the right
given by
\[
\Phi(A)*B= \Phi(AB)
\]
\noindent For $\gamma\in\Gamma$ we have
\[
\Phi_{\Gamma} (A)*B= \Phi_{\Gamma} (AB) = \Phi_{\Gamma} (\gamma AB) =
\gamma\cdot\Phi_{ \Gamma}(AB) =\gamma \cdot \left(\Phi_{\Gamma}(A)*B\right)
\]

\noindent so the action is well defined on $\X $. All three flows are
special cases, eg. write
\[
G^{t} = \left(
\begin{array}
[c]{cc}%
e^{\frac{t}{2}} & 0\\
0 & e^{-\frac t2}\\
\end{array}
\right)  , \quad U_+^t = \left(
\begin{array}
[c]{cc}%
1 & t\\
0 & 1\\
\end{array}
\right) , \quad  U_{-}^{t} = \left(
\begin{array}
[c]{cc}%
1 & 0\\
t & 1\\
\end{array}
\right)
\]
and name the corresponding one-parameter subgroups
\[
G=\{G^{t} |\ t \in\R\},\quad U_+=\{U_+^{t} |\ t \in\R\},
\quad U_-=\{U_-^t |\ t \in\R\}.
\]
Then
\[
\x g(t)=\x*G^{t}, \quad \x u_{+}(t)=\x*U_{+}^{t}, \quad \x u_{-}(t)=\x*U_{-}^{t}
\]

Let us check these. By naturality we see that for all $A \in \PSL_{2}\R$ we have
\[
(A \cdot\mathbf{x_{0}}) g(t) = A \cdot(\mathbf{x_{0}} g(t)),\ (A
\cdot\mathbf{x_{0}}) u_{+}(t) = A \cdot(\mathbf{x_{0}} u_{+}(t)),\ (A
\cdot\mathbf{x_{0}}) u_{-}(t) = A \cdot(\mathbf{x_{0}} u_{-}(t))
\]
and, moreover
\[
\Phi(G^{t}) = \x_0 g(t), \;\;\;\; \Phi(U_+^t) = \x_0
u_+(t), \quad \Phi(U_-^t)=\x_0 u_{-}(t)
\]
\noindent so
\[
\Phi(A G^{t}) = (A G^{t}) \cdot\mathbf{x_{0}} = A \cdot(G^{t} \cdot
\mathbf{x_{0}}) = A \cdot(\mathbf{x_{0}} g(t)) = (A \cdot\mathbf{x_{0}}) g(t)
= \Phi(A) g(t)
\]
\noindent and the argument for $u_{+}$ and $u_-$ is identical.\\

Left multiplication by $G^{t},\ U_+^{t}$ and $U_-(t)$ also give flows on $\H$; probably easier to understand than the geodesic and horocycle flows. For instance, left action by $U_+^{t}$ corresponds to translating a point and vector to the right by $t$. But these actions do not commute with the action of $\Gamma$ and hence induce nothing on $\X $.

Since $\PSL_2\R$ is unimodular, it has a Haar measure, invariant under both left and right translation, and unique up to
multiples.  Since $\PSL_2\R$ is not compact, there is no natural normalization.  Denote by $\omega$ the corresponding measure on
$\H$; if $\X=\Gamma\backslash \H$ is of finite volume, we will denote by $\omega_\X$ the corresponding measure normalized so that
$\omega_\X(\X)=1$. Up to a constant multiple we have
$$
\omega=\frac {dx\wedge dy\wedge d\theta}{y^2},
$$
where we have written $\x=(z,\xi)$ and $z=x+iy$, $\xi= ye^{i\theta}$ (the factor $y$ is there to make it a unit vector): this
measure is easily confirmed to be invariant under both left and right action of $\PSL_2\R$ on $\H$.

Occasionally, we will need a metric and not just a measure on $\PSL_2\R$; we will use the metric that corresponds under
$\Phi$ to the Riemannian structure
$$
\frac {dx^2+dy^2}{y^2} +d\theta^2.
$$
This metric is invariant under left action of $\PSL_2\R$ on $\H$, and as such does induce a metric on $\X$.  It is {\it not
invariant} under right action, and the flows $u_+, u_-$ and $g$ do not preserve lengths.

\section{The Horocycle Flow is Ergodic}

\begin{theorem}\cite{Hedlund36}
\label{Hedlund} The positive and the negative horocycle flows  are ergodic.
\end{theorem}

We will show this for the positive ergodic flow. To prove Theorem 1 we will show that any $f \in L^2(\X )$
invariant under the horocycle flow is constant almost everywhere. Indeed, if
the positive horocycle flow is not ergodic then then there is a measurable set
$\Y \in\X $ of positive but not full measure that is invariant
under $U_{+}$ and the characteristic function $\charone_{\Y}$ provides
a nonconstant invariant element of $L^2(\X)$.

\begin{lemma} \label{biinvariantlemma}
For $f \in L^{2}(\X )$, $A \in \PSL_{2}\R$ and $\x
\in\X $, let $(T_{A} f)(\x) := f(\x * A)$. Then the
function $F_f: \PSL_{2}\R \rightarrow\mathbb{R}$ defined by
\[
F_f(A)=\int_{\X } f(\x)f(\x*A) \omega_\X(d \x) := < f,T_{A}f>
\]
is

(a) uniformly continuous and

(b) bi-invariant under $U_{+}$, i.e., invariant under the left and the right action of $U_+$ on $\PSL_2\R$.
\end{lemma}

\textbf{Proof of Lemma \ref{biinvariantlemma}} (a) Choose $\varepsilon> 0$. Since the continuous functions with
compact support are dense in $L^{2}(\X )$, we can find a function $g
\in C_{c}(\X )$ with $||f-g||_{2} < \varepsilon/3$. The fact
that such a $g$ is uniformly continuous means that $\exists\delta>0$ such
that
\[
d(A,B)< \delta\Rightarrow||T_{A}g-T_{B}g||_{2} < \frac{\varepsilon}{3}
\]
So when $d(A,B) \leq\delta$ we have
\[
\|T_{A} f - T_{B} f\|_{2} \leq\|T_{A} f - T_{A} g\|_{2} + \|T_{A} g - T_{B}
g|_{2} + \|T_{B} g - T_{B} f\|_{2} \leq\varepsilon
\]
We see that $A \mapsto T_{A}f$ is a uniformly continuous map
$\PSL_{2}\R \rightarrow L^{2}(\X )$ and (a) follows.

\medskip
(b) Biinvariance reflects the invariance of Haar
measure on $\PSL_2\R$ under left and right translation: for $A \in \PSL_{2}\R$ we have
\begin{align*}
F_f(A U_{+}^{t})&=\int_{\X } f(\x) f(\x*(AU_+^t)) \omega_\X(d \x) =
\int_{\X } f(\x) f((\x*A) u_{+}(t))\omega_\X(d \x)\\
& = \int_{\X } f(\x) f(\x *A) \omega_\X(d \x) = F_f(A);\\
F_f(U_{+}^{t} A)&=\int_{\X } f(\x) f(\x *(U_{+}^{t}A )) \omega_\X(d \x)  = \int_{\X } f(\x *A^{-1}) f(\x *U_{+}^{t}) \omega_\X(d \x)   \\
&= \int_{\X } f(\x *A^{-1}) f(\x  u_{+}(t) \omega_\X(d \x) =
\int_{\X } f(\x *A^{-1}) f(\x ) \omega_\X(d \x) \\
&=
\int_{\X } f(\x ) f(\x  *A) \omega_\X(d \x) =F_f(A)
\end{align*}

\begin{figure}[!ht]
\begin{centering}
\includegraphics[scale=0.40]{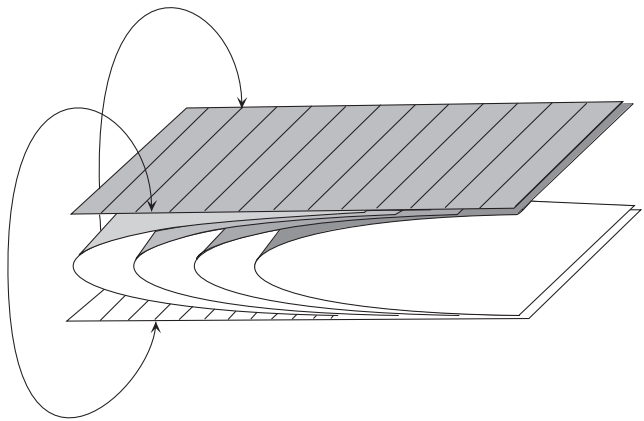}
\caption {\label{biorbits1} Since the left and the right action of $U_+$ commute, the biorbits are homeomorphic to $\R^2$, except
the orbits on which the two actions coincide. Viewed in $\H$, the orbits are the 1-parameter family of folded planes (the
ham slices in the sandwich). The top and the bottom planes should be identified; they represent the orbits formed of vertical
upwards pointing vectors; those orbits are lines, as drawn in the planes.  The salient features of the figure is that any two points
of the top plane are within $\epsilon$ of a single biorbit (in fact all of those with folds sufficiently far to the left),  and
every biorbit comes arbitrarily close to a biorbit consisting of vertical upwards pointing vectors.  }
\end{centering}
\end{figure}

\begin{proof} [Proof of Theorem 1] What do the biorbits of $U_+$ look like in $\PSL_2\R$? Using our identification $\Phi: \PSL_2\R\to \H$, we can think of the biorbits  as living in $\H$ with geometry  represented in figure \ref{biorbits1}.  More
specifically, there are two kinds of biorbits.  The first (exceptional) kind of biorbit consists of all upwards pointing vertical vectors anchored at points $z=x+iy$ with a given $y$-coordinate. The union of these orbits forms a plane $\V\subset \H$.  The other (generic) biorbits consist of the vectors defining horocycles of a given radius tangent to the $x$-axis: each such biorbit is
diffeomorphic to a plane.

In particular, all the biorbits are closed, and there is nothing to prevent the existence of nonconstant continuous functions on
$\H$ that are constant on biorbits. But our function $F_f$ is {\it uniformly} continuous, and that changes the situation: every
uniformly continuous function on $\H$ that is constant on biorbits {\it is} constant. What we need to see is that for every
$\epsilon>0$,\\

\noindent$\bullet$    any two elements of $\V$ can be approximated to within $\epsilon$ by a single 2-dimensional biorbit, and\\

\noindent$\bullet$  that any 2-dimensional biorbit is within $\epsilon$ of some element of $\V$. \\

\noindent These features are illustrated, but not proved, by figure \ref{biorbits1}.  The proofs are the content of the two parts of figure \ref{almostverticalfig}:\\

\begin{figure}[!ht]
\begin{centering}
\includegraphics[scale=0.50]{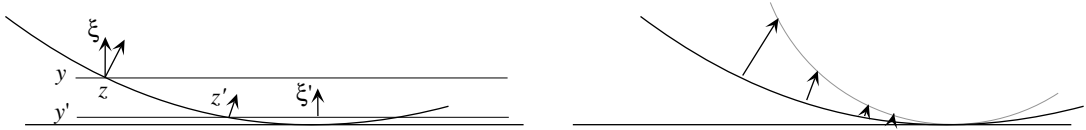}
\caption {\label{almostverticalfig} Left: Any two upward-pointing vertical unit vectors $\xi, \xi'$ can be approximated by
elements of the same biorbit. Right: Any orbit contains vectors arbitrarily close to upward-pointing vertical vectors.}
\end{centering}
\end{figure}

 Since $F_f$ is uniformly continuous, for any $\epsilon >0$ we can find a $\delta$ such that $d(\z,\z')<\delta \Rightarrow \|F_f(\z)-F_f(\z')\| < \epsilon/2$.  Choose any two points $\z=(x+iy, \xi), \z'=(x'+iy', \xi') \in \V$, and assume without loss of generality that $y>y'$.  Choose $\eta$ a \textit{nonvertical} vector for which $d((x+iy, \xi),(x+iy,\eta)) < \delta$.  Set $\w''=(x''+iy',\eta')$ is the point on the the positive horocycle through $\w=(x+iy,\eta)$ at ``height'' $y'$, and $w'=(x'+iy',\eta')$.  The vector $\eta'$ is \textit{more} vertical than $\eta$, so $d((x''+iy',\eta'),(x''+iy',\xi')) <\delta$.  Further, $w'$ and $w'''$ belong to the same  biorbit.  Thus
  \begin{align*}
 &\|F_f(\z)-F_f(\z')\| \\
 & \leq \|F_f(\z)-F_f(\w)\|+\|F_f(\w)-F_f(\w'')\|+\|F_f(\w'')-F_f(\w')\|&+\|F_f(\w')-F_f(\z')\|\\
 &\leq \frac \epsilon 2 + 0 + 0 + \frac \epsilon 2 = \epsilon.
\end{align*}
The right side of fig \ref{almostvertical} shows that every two-dimensional biorbit contains vectors arbitrarily close to vertical, ie. arbitrarily close to $\V$ just taking the vector perpendicular to the horocycle sufficiently close to the real axis.

Using biinvariance, we see that $F_f$ is constant on $\V$; evidently this constant is $\|f\|_2^2= F_f(\x_0)$.  By the Cauchy-Schwarz inequality (and the fact that we chose $f$ real), we see that for all $A \in \PSL_2(\R)$, $T_A f=\pm f$ in $L^2(\X)$ with the sign depending continuously on $A$. Since $\PSL_2(\R)$ is connected, $f$ is a constant element of $L^2(\X)$ and the
theorem follows.
\end{proof}

\section {Equidistribution of the horocycle flow}

\bigskip

The main result of this paper is theorem \ref{Ratnersthm}.

\begin{theorem} \label{Ratnersthm} \cite{DaniSmillie84}
Let $X$ be a complete hyperbolic surface of finite area. Then every horocycle
on $\X $ is either periodic or equidistributed in $\X $.
\end{theorem}

\medskip

This theorem is evidently a much stronger statement than \ that the horocycle
flow is ergodic, or even that it is uniquely ergodic. \ It is not an ``almost
everywhere''statement, but rather it asserts that \textit{every horocycle} is
either periodic or equidistributed in $\X $.

Note that this depends crucially on the fact that horocycles have geodesic curvature 1.  The statement is false for
geodesics (geodesic curvature 0): geodesics can do all sorts of things other than being periodic or equidistributed.
For instance, they can spiral towards a closed geodesic, or be dense in a geodesic lamination, or spiral towards a
geodesic lamination. Curves with constant geodesic curvature $<1$ stay a bounded distance away from a geodesic, and
hence can do more or less the same things as geodesics; in particular, they do not have to be periodic or
equidistributed.

On the other hand, curves with geodesic curvature $>1$ are always periodic, hence never equidistributed.

The horocycle flow is also distinguished from flows along curves of geodesic curvature less than 1 by having entropy 0.  We will not define entropy here, but whatever definition you use it is clear that if you speed up a flow by a factor $\alpha>0$, the entropy will be multiplied by $\alpha$.  But the formula

$$
G(-s)U^{+}(t)G(s)=U^+(\exp(-s)t)
$$

shows that the horocycle flow is conjugate to itself speeded up by $\exp(-s)$, thus its entropy must be 0.

\section{The Geometry of Flows in $\H$}

The geodesic flow in $\X$ has {\it stable} and {\it unstable} foliations: two points $\x_1, \x_2\in \X$ belong to the
same leaf of the stable foliation if $d(\x_1g(t),\x_2g(t))$ is bounded as $t \to +\infty$, and they belong to the same
leaf of the unstable foliation if $d(\x_1g(t),\x_2g(t))$ is bounded as $t \to -\infty$.  These foliations are very easy
to visualize in $T^1\H$, as shown in Figure \ref{stablemanifolds}.

\begin{figure}[!ht]
\begin{centering}
\includegraphics[scale=0.5] {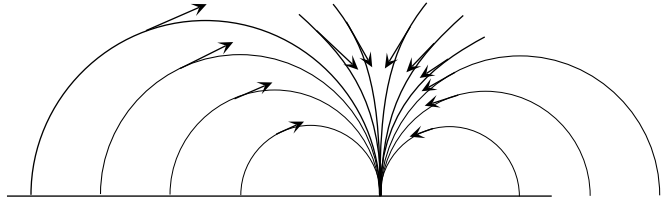}
\caption {\label  {stablemanifolds}The tangent vectors to geodesics ending at the same point at infinity form one leaf of the stable foliation for the geodesic flow. Similarly, the tangent vectors to geodesics emanating from a
point at infinity form a leaf of the unstable foliation.  In $\X$, these leaves are tangled up in some very complicated
way (after all, most geodesics are dense, never mind their stable and unstable manifolds). But clearly each leaf is an
immersed smooth surface, hence of measure $0$ in $\X$.}
\end{centering}
\end{figure}

Note that the stable leaves are fixed by the positive horocycle flow: the positive horocycles are the
curves orthogonal to the geodesics in a leaf.  Similarly, the unstable manifolds are fixed by the negative geodesic
flow, and the negative horocycles in a leaf are the curves orthogonal to the geodesics in that leaf.

On the other hand, the positive horocycles are transverse to the unstable manifolds, and positive horocycle flow does
not send unstable leaves to unstable leaves.

Let $S_{\x }$ be the unstable manifold of the geodesic flow through $\x$. Define $S_{\x }(a,b)\subset S_\x$ by
$$
S_\x(a,b)=\{\x u_-(r)g(s), |r|\le a, |s| \le b\}.
$$
We will refer to $S_\x(a,b)$ as a ``rectangle''; it isn't really: it is a quadrilateral bounded by two arcs of geodesic
of length $2b$, and be two arcs of negative horocycle, of length respectively $ae^b$ and $ae^{-b}$ (see figure
\ref{SxandVx}).

Further we define the ``box''
$W_\x (a,b,c)\subset \X$ as the region obtained by flowing along positive horocycles from $S_\x(a,b)$ until you hit
$S_{\x u_+(c)}$. Because $S_{\x u_+(c)}$ is actually dense in $\X$, you have to understand the flow as taking place in
the universal covering space $\H$, and then projecting the ``box'' to $\X$ (see figure \ref{SxandVx} again).

Each surface $S_\x$ is invariant under geodesic flow, but the
``rectangles''  $S_{\x }(a,b)$ are not; instead we have
$$
S_{\x }(a,b) g(s)=S_{\x  g(s)}(e^{s} a, b).
$$

Moreover, the surfaces $S_{\x u_+(s) }$ foliate a neighborhood of the positive
horocycle $\x * U_+$ through $\x $, and thus there
is a function $\alpha_\x(\y,t): \R \to
\mathbb{R}$ for $\mathbf{y} \in S_{\x }$ (for its precise domain, see below) such that
$$
\y u_{+}(\alpha_{\x }(\y, t)) \in S_{\x  u_{+}(t)}.
$$

\bigskip

\noindent as sketched in Figure \ref{SxandVx}. \newline\begin{figure}[ptb]
\begin{centering}
\includegraphics[scale=0.65]{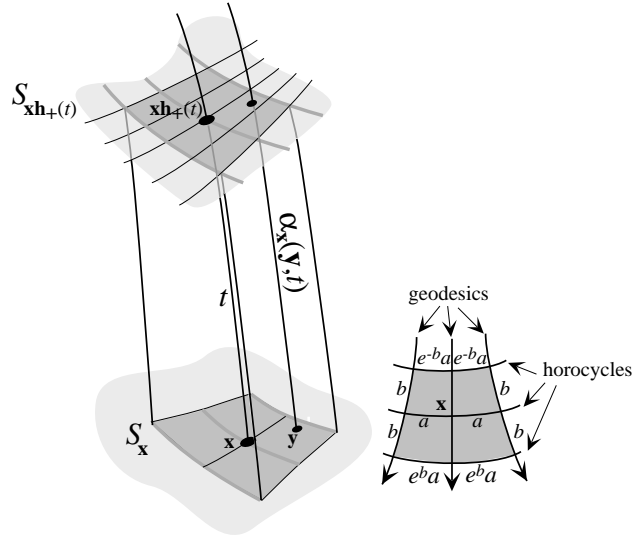}
\caption{\label {SxandVx} The surfaces $S_{\x u_+(s)}$ foliate a neighborhood of the positive horocyclic orbit
 of $\x $.  Thus, for every $t$ and every $\y \in S_\x$ sufficiently close to $\x$, there is a time
$\alpha_\x
(\y, t)$  such that the positive horocycle $\y u_{+}(\R)$ intersects $S_{\x u_+(t)}$.}
\end{centering}
\end{figure}

The function $t \mapsto \alpha_\x(\y,t)$ is defined in $[0, T(\y))$ for some $T(\y)$ that
tends to $\infty$ as $\y\to \x$.  Moreover, the function is $C^\infty$ (actually real-analytic) by the implicit
function theorem, and
$\frac d{dt} \alpha_\x(\y,t)$ tends to $1$ as $\y \to \x$, so that
$$
\lim_{\y \to \x} \frac {\alpha_\x(\y,t)}t \to 1.
$$

It isn't often that you can replace the implicit function theorem by an explicit formula, but this does occur here.

\begin{lemma}\label{explicitalpha}  If $\y= \x u_-(r)g(t)$, then
\begin{equation} \label{explicitalpha}
\alpha_\x(\y,s) = \frac s{e^t(1-rs)}.
\end{equation}

In particular, $\alpha_\x(\y,s)$ is defined in  $\y\in S_\x(a,b)$ for all $s<1/a$ and all $b$, and
$$
\lim_{\y \to \x} \frac{d}{ds}\alpha_\x(\y,s)=1,\quad \lim_{\y \to \x} \frac{\alpha_\x(\y,s)}s=1.
$$
\end{lemma}

\begin{proof} This is a matter of solving the equation
$$
\x u_-(r)g(t)u_+(\alpha_\x(y,s))=\x u_+(s)u_-(\rho)g(\tau),
$$
i.e., the matrix equation
$$
\bmatrix 1&0\\r&1\endbmatrix  \bmatrix e^{t/2}&0\\0&e^{-t/2}\endbmatrix \bmatrix 1&\alpha\\0&1\endbmatrix=
\bmatrix 1&s\\0&1\endbmatrix   \bmatrix 1&0\\\rho&1\endbmatrix
 \bmatrix e^{\tau/2}&0\\0&e^{-\tau/2}\endbmatrix.
$$
This is a system of 3 equations (because the determinants are all 1) in 3 unknowns $\rho, \tau$ and $\alpha$.  Just
multiply out and check.
\end{proof}
\bigskip

 The central result here is the following:

\medskip

\begin{lemma}\label{horokeeptogether}
There exists a constant $C$ such that for all $0< \delta< 1/2$, all $t > 0$, all
$\y \in S_{\x }(\frac{\delta}{t},\delta)$ and all $0 \leq s
\leq t$ we have
$$
d(\x  u_{+}(s), \y u_{+}(\alpha_\x(\y, s)) \leq C \delta .
$$
\end{lemma}

Note that $\alpha_\x (\y ,s)$ is defined for $s\le t$ when $\y \in S_{\x }(\frac{\delta}{t},\delta)$ and $\delta\le
1/2$, since for the factor $1-rs$ from the denominator of formula \ref{explicitalpha}, we have $r\le \delta/t$ and
$s\le t$, so $1-rs \ge 1-\delta^2=3/4$.

\begin{figure}[!ht]
\begin{centering}
\includegraphics[scale=0.55]{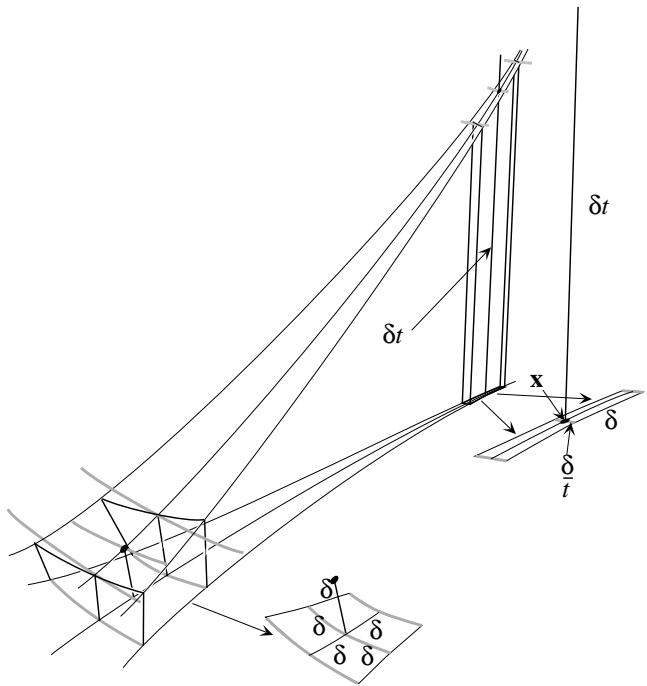}
\caption{\label{skinnyandstandardboxes} The skinny ``rectangle'' $S_{\x }(\frac{\delta}{t},\delta)$
becomes under the geodesic flow for time $\log{t}$ the ``square'' $S_{\x  g(\log{t})}(\delta, \delta)$,
and the box $W_{\x }(\frac{\delta}{t},\delta, \delta t)$ becomes the box $W_{\x g(\log t) }(\delta, \delta, \delta)$.
The geometry of $W_{\x g(\log t) }(\delta, \delta, \delta)$ is standard: it depends only on
$\delta$.  In particular, the positive horocyclic flow from the bottom
$S_{\x }(\frac{\delta}{t},\delta)$ of the box is defined since $\delta \leq 1/2$, hence $C^\infty$, hence Lipschitz
with a universal constant $C$.}
\end{centering}
\end{figure}

\begin{proof}
The proof essentially consists of gazing at Figure \ref{skinnyandstandardboxes}. Almost everything in
that figure comes from the fact that the geodesic flow takes horocycles to
horocycles; moreover, geodesic flow for time $t$ maps a segment of positive
horocycle of length $l$ to one of length $e^{-t} l$, and a segment of
negative horocycle of length $l$ to one of length $e^{t} l$.

Two points $\x,\y$ with $\y \in S_\x(\delta/t,\delta)$ flow under the geodesic flow for time $\log t$ to two
points
$\x'=\x g(\log t)$ and $\y'=\y g(\log t)$; note that $\y'\in S_{\x'}(\delta,\delta)$, so certainly $d(x',y')\le
2\delta$. Then under positive horocycle flow (for different times) these points flow to  points
$$
\x''= \x'u_+(s\delta) \quad \text{and}\quad \y''\in S_{\x''}.
$$
By the argument in the caption of figure \ref{skinnyandstandardboxes}, there exists a universal constant $C$
such that
$d(\x'',\y'')\le 2
Cd(\x,\y)$. Finally, use the geodesic flow back, i.e., for time $-\log t$, to find points
$$
\x'''= \x u_+(s),\ \y'''= \y u_+(\alpha_\x(\y,s)).
$$
Since backwards geodesic flow in a single unstable manifold is contracting, we find $d(x''', y''')\le 2C\delta$.

\end{proof}

\section{Geometry of hyperbolic surfaces and cusps}

Let $X$ be a complete hyperbolic surface. If such a surface is not compact, it has finitely many {\it cusps}. Every
cusp
$c$ is surrounded by closed horocycles, and the open region bounded by the horocycle of length 2 is
a neighborhood $N_{c}$ isometric to the region $ 2\Z \backslash \{y\ge 1\}$  that is embedded in $X$, moreover, if $c$,
$c'$ are distinct cusps, then $N_{c} \cap N_{c'} = \emptyset$.  If $\X$ has finite area and since each of the disjoint neighborhoods $N_c$ of the cusps has area 2, there are only finitely many cusps. \cite{Hubbard06}.

We know everything about the standard cusp $ \{y\ge 1\}/2\Z$, in particular that any geodesic that enters it will leave it again unless it goes directly to the cusp, i.e., unless it is a vertical line. Thus the same holds for all $N_c$.

\goodbreak

\begin{figure}[!ht]
\begin{centering}
\includegraphics[scale = 0.5] {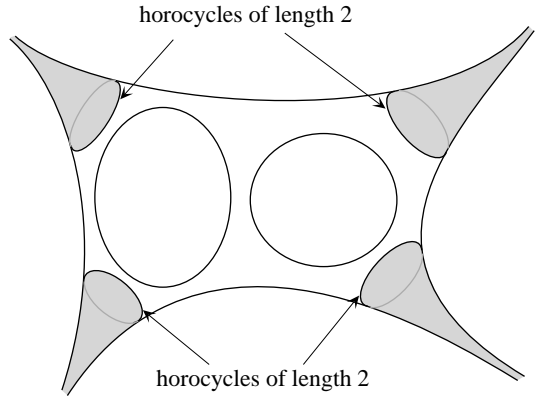}
\caption{\label{surfacewithcusps} A non-compact complete hyperbolic surface is always non-compact in the same
way: it has cusps
 $c$ with disjoint standard neighborhoods  $N_c$ isometric to   $ \{y\ge 1\}/2\Z$, hence bounded by horocycles
of length 2.  Note that the only way a geodesic $\gamma(t)$ can stay in such a neighborhood for all $t\ge t_0$ is to
head straight to the cusp. Each cusp has a stable manifold in $X$, and the geodesics that do not return infinitely
many times to $X_c= X-\cup_c N_c$ are those that belong to one of these stable manifolds.}

\end{centering}
\end{figure}

If  $X$ has finite area, then the complement of these neighborhoods is compact:
\[
X_c= X-\bigsqcup_{\text{cusps $c$ of $X$}} N_{c}
\]
is a compact set. Denote by $\X_c$ the corresponding part of
$\X$. The injectivity radius is bounded below on $\X_c$, so  there is a number $\delta_{\X } > 0$ such that
for every $\x \in\X_c$  the box
$W_{\x }(\delta_\X ,\delta_\X ,\delta_\X)$ is embedded in $\X$.

Now, suppose that $\X $ has finite measure. Then if $(x,\xi )
\in\X $ is a point through which the positive horocycle is not periodic, the
geodesic through $(x,\xi )$ does not go forward to a cusp and hence must
enter $\X_c$ infinitely many times.\newline

All periodic horocycles are homotopic to horocycles surrounding cusps.  Indeed, if we apply geodesic flow to a positive horocycle, it will become arbitrarily short, thus will either be contained in a neighborhood of a cusp or in a contractible subset of $\X$.  No horocycle in such a contractible subset is closed, thus the horocycle is homotopic to a horocycle surrounding a cusp. 

Thus, the set $\P_\X\subset \X$ of points defining periodic positive horocycles is the union of the stable
manifolds of the cusps (for the geodesic flow). 
\begin{lemma}
The set $\P_\X$ has measure zero in $\X$ for the measure $\omega_\X$.
\end{lemma}

\begin{proof} There are finitely many cusps, and each has a stable manifold which is a smooth immersed surface, certainly of
3-dimensional measure $0$.
\end{proof}

Although we have used the fact that $\X$ has finite area, the result is true for every complete hyperbolic surface, since such a surface can have only countably many cusps.

\section{A sequence of good times} In this section we prove a result, still a bit weaker than theorem \ref{Ratnersthm},
though it does prove theorem \ref{Ratnersthm} when $X$ is compact.

\begin{theorem} \label{goodtimesthm} Let $X$ be a complete hyperbolic surface of finite area, $\X$ be its unit tangent bundle.
For all  $\x\notin \P_\X$, there then exists a sequence $T_n\to \infty$ such that for any function $f\in C_c(\X)$ we have
$$
\lim_{n\to \infty} \frac 1{T_n} \int_0^{T_n} f(\x u_+(t)) dt = \int_\X f d\omega_\X.
$$
\end{theorem}

The proof will take the remainder of this section.

\medskip

\begin{proof} Let $T_n$ be any increasing sequence tending to $\infty$ such that $\x g(t) \in \X_c$.  Such a sequence exists because $\x \notin \P_\X$.

Choose $\epsilon>0$, and  $f\in C_c(\X)$; without loss of generality we may assume $\sup |f|= 1$ and that
$\epsilon<1$ .

We have already defined $\delta_\X$.  We need two more $\delta's$, to be specified in lemmas \ref{deltaf} and \ref{deltaalpha}.

\begin{lemma}\label{deltaf}  There exists $\delta_f>0$ such that for all $t>0$, if $\z \in S_\x(\delta_f/t,\delta_f)$ and $0\le s\le
t$, then
$$
 |f(\x u_+(s))-
f(\z u_+(\alpha_\x(\z,s)))|<\epsilon.
$$
\end{lemma}

\begin{proof} This follows immediately from proposition \ref{horokeeptogether} and the uniform continuity of $f$.
\end{proof}

\begin{lemma} \label{deltaalpha} There exists $\delta_\alpha$ such that  for all $t>0$, if $\z \in S_\x(\delta_\alpha/t,\delta\alpha)$ and $0\le s\le t $, then
$$
|\alpha'_\x(\z,s)-1|<\epsilon.
$$
\end{lemma}

\begin{proof} One could derive this from the implicit function theorem, but we might as well use our explicit
formula (\ref {explicitalpha}) for  $\alpha$.  For $\z =\x u_-(r)g(u) \in
S_\x(\delta_\alpha/t,\delta_\alpha)$ we have
$$
\alpha_\x'(\y ,s)= \frac 1{e^u(1-rs)^2},
$$
and since $|r|\le \delta_\alpha/t,\ |u|\le \delta_\alpha$ and $s\le  t$,
$$
\frac 1{e^\delta_\alpha (1+\delta_\alpha)^2} \le \alpha_\x'(\y ,s) \le \frac {e^\delta_\alpha}{(1-\delta_\alpha)^2}.
$$
Clearly we can choose $\delta_\alpha$ so that
$$
\left|\frac 1{e^{\delta_\alpha}(1+\delta_\alpha)^2}-1\right|<\epsilon, \quad  \left| \frac
{e^{\delta_\alpha}}{(1-\delta_\alpha)^2}-1\right|<\epsilon.
$$
\end{proof}

\goodbreak

\bigskip

Set $\delta=\inf(\delta_\X, \delta_f, \delta_\alpha)$, and $\eta= \omega_\X(W_\x(\delta, \delta, \epsilon\delta))$.

\begin{proposition}\label{EgorovAndBirhoff}
There exists a $\tilde{T}$ and a set $\Y\subset \X$ with $\omega_\X(\Y)>1-\eta$ such that for all
$T\ge \tilde{T}$ and all $\y\in\Y$ we have
$$
\left|\frac{1}{T}\int_{0}^{T}f(\y u_{+}(t))dt-\int_{\X } fd\omega_\X\right|<\epsilon
.$$
\end{proposition}

\begin{proof} By the ergodic theorem, the family of functions
$$
g_T(\y)= \frac{1}{T}\int_{0}^{T}f(\y u_{+}(t))dt
$$
converges almost everywhere as $T \to \infty$, and since horocycle flow is ergodic, it converges almost everywhere to $\int_{\X} fd\omega_\X$ (this is where we use Theorem \ref{Hedlund}). By Egorov's theorem, there exists a set $\Y$ of measure at least $1-\eta$ such that the $g_T$ converge uniformly on $\Y$; omitting a set of measure $0$ from $\Y$, the family $g_T$ converges uniformly to $\int_{\X} fd\omega_\X$ on $Y$. This is the assertion of Proposition \ref {EgorovAndBirhoff}.
\end{proof}

\bigskip
We can now choose
\begin{enumerate}
\item an $n_0$ such that $T_n>\tilde{T}$ for $n>n_o$.  Then for $n>n_0$ we can select:

\item a sequence $\y_n \in \Y \cap  W_\x(\delta/T_n, \delta, \epsilon\delta T_n)$. Indeed, we have
$$
\omega_\X(W_\x(\delta/T_n, \delta, \epsilon\delta T_n))= \eta,
$$
since it is the inverse image  of $W_{\x u_+(T_n)}(\delta, \delta,
\epsilon\delta)$ by the geodesic flow at time $T_n$.  We have
$$
\omega_\X(W_{\x u_+(T_n)}(\delta, \delta, \epsilon\delta))=\eta
$$
since  $\x u_+(T_n)\in \X_c$ and $\delta\le
\delta_\X$ . Geodesic flow for a fixed time is a measure-preserving diffeomorphism, so $W_\x(\delta/T_n, \delta, \epsilon\delta
T_n)$ must intersect $\Y$ which has volume $>1-\eta$.
\item sequences $\z_n\in S_\x(\delta/T_n, \delta)$ and $\epsilon'_n\le \epsilon$ such that
$\z_n u_+(\epsilon_n'\delta T_n)= \y_n$. This is just what it means to say $\y_n \in W_\x(\delta/T_n, \delta,
\epsilon\delta T_n)$.
\end{enumerate}
\bigskip
The organizing principle is now to write for $n>n_0$
\begin{align*}
&\left|\frac 1{T_n}\int_0^{T_n} f(\x u_+(t))dt- \int_\X f(\w) \omega_X(d\w)\right| \le  \\
&\qquad \left|\frac 1{T_n} \int_0^{T_n} f(\x u_+(t))dt- \frac 1T_n \int_0^{T_n} f\Bigl(\z_n u_+\bigl(\alpha_\x(\z_n,t)\bigr)\Bigr)\right|+\\
&\qquad \left|\frac 1T_n \int_0^{T_n} f\Bigl(\z_n u_+\bigl(\alpha_\x(\z_n,t)\bigr)\Bigr)- \frac 1{T_n} \int_0^{T_n} f\Bigl(\z_n
u_+\bigl(\alpha_\x(\z_n,t)\bigr)\Bigr) \alpha'_\x(\z_n,t)dt\right|+\\ &\qquad \left|\frac 1{T_n} \int_0^{\alpha_\x(\z_n, T_n)} f(\z_n u_+(s) )ds-
\frac 1{T_n} \int_0^{T_n} f(\z_n u_+(s))ds\right|+\\
&\qquad  \left|\frac 1{T_n} \int_0^{T_n} f(\z_n u_+(s))ds- \frac 1{T_n} \int_0^{T_n} f(\y_n u_+(s)) ds\right| +\\
&\qquad\left|\frac 1{T_n} \int_0^{T_n} f(\y_n u_+(s)) ds-\int_\X f(\w) \omega_\X(d\w)\right|.
\end{align*}

To get from the second summand on the right to the third, we use the change of variables formula, setting $s= \alpha_\x(\z_n,t))$:
$$
\int_0^{T_n} f\Bigl(\z_n u_+\bigl(\alpha_\x(\z_n,t)\bigr)\Bigr) \alpha'_\x(\z_n,t) dt= \int_0^{\alpha_\x(\z_n, T_n)} f(\z_n u_+(s) )ds.
$$

\medskip
Each of
the five terms above needs to be bounded in terms of
$\epsilon$.

\begin{enumerate}
\item
Since  $\delta<\delta_f$, we have $|f(\x u_+(t))-f(\z_n u_+(\alpha_\x(\z_n,t)))|<\epsilon$, so
$$
\left|\frac 1{T_n} \int_0^{T_n} f(\x u_+(s))ds- \frac 1T_n \int_0^{T_n} f(\z_n u_+(\alpha_\x(\z_n,t))) dt\right|<\epsilon.
$$
\item
Since $\delta < \delta_\alpha$, we have
\begin{align*}
&\left|\frac 1{T_n} \int_0^{T_n} f(\z_n u_+(\alpha_\x(\z_n,t))) dt- \frac 1{T_n} \int_0^{T_n} f(\z_n
u_+(\alpha_\x(\z_n,t))) \alpha'_\x(\z_n,t)dt\right|\\
&\le \frac 1{T_n}\int_0^{T_n} \sup |f|\ |1-\alpha'_\x(\z_n,t)| dt < \epsilon.
\end{align*}
\item
From $\delta<\delta_\alpha$,
so $|\alpha'-1|< \epsilon$,  we get that $(1-\epsilon)T_n<\alpha_\x(\z_n,T_n) < (1+\epsilon)
T_n$ and hence
$$
\left|\frac 1{T_n} \int_0^{\alpha_\x(\z_n, T_n)} f(\z_n u_+(s) )ds-
\frac 1{T_n} \int_0^{T_n} f(\z_n u_+(s))ds\right|<\epsilon.
$$
\item
The points $\z_n$ and $\y_n$ are on the same positive horocycle, a distance $\epsilon_{n}' T_n$ apart for some $\epsilon_{n}' \le
\epsilon$. This leads to

\begin{align*}
&\left|\frac 1{T_n} \int_0^{T_n} f(\z_n u_+(s))ds- \frac 1{T_n} \int_0^{T_n} f(\y_n u_+(s)) ds\right|\\
&=\left|\frac 1{T_n} \int_0^{T_n} f(\z_n u_+(s))ds- \frac 1{T_n} \int_{\epsilon_{n}' T_n}^{(1+\epsilon_n')T_n} f(\z_n u_+(s))
ds\right|\le \frac {2\epsilon T_n}{T_n} = 2\epsilon.
\end{align*}
\item Since $\y_n\in \Y$ and $T_n>\tilde{T}$, we have
$$
\left|\frac 1{T_n} \int_0^{T_n} f(\y_n u_+(s)) ds-\int_\X f(\w) \omega_\X(d\w)\right|<\epsilon.
$$
\end{enumerate}

\end{proof}
This ends the proof of theorem \ref{goodtimesthm}.\qquad  $\square$

\section{Proving equidistribution}  The  $T_n$ are chosen to be a sequence of times tending to infinity such that $\x g(T_n) \in
\X_c$.  Thus if $\X$ is compact, the sequence $T_n$ is an arbitrary sequence tending to infinity, and so equidistribution is proved
in that case. Moreover, clearly theorem \ref{goodtimesthm} shows that all non-periodic horocycles are dense in $\X$.  But it
doesn't quite prove that they are equidistributed when $\X$ is not compact; perhaps a horocycle could spend an undue amount
of time near some cusp, and we could choose a different sequence of times $T'_n$ also tending to infinity which would
emphasize the values of $f$ near that cusp. In fact, we will see in Section 9 that something like that does happen for random
walks on horocycles.

We will now show that this does not happen for the horocycle flow itself.

\begin{proposition} \label{classmeasures} Let $\nu$ be a probability measure on $\X$ invariant under the positive horocycle flow and such that
$\nu(\P_\X)=0$ .  Then $\nu=\omega_\X$.
\end{proposition}

\begin{proof} Without loss of generality, we can assume that $\nu$ is ergodic for the positive horocycle flow since any
invariant probability measure is a direct integral of ergodic invariant probability measures $\nu_\alpha$ with $\nu_\alpha(\P_\X)=0$, so
uniqueness for such invariant ergodic measures implies uniqueness for all such invariant measures.  Choose $f\in C_c(\X)$, and let $\x\in \X$
be a typical point for $\nu$, i.e., a point of $\X-\P_\X$ such that
$$
\int_\X f d\nu = \lim_{t\to \infty} \frac 1t \int_0^t f(\x u_+(s)) ds.
$$
By the ergodic theorem, this is true of $\nu$-almost every point, so such points $\x$ certainly exist. Such a point is one for
which the horocycle flow is not periodic, so theorem \ref{goodtimesthm} asserts that there exists a sequence $t_n\to \infty$ such
that
$$
\int_\X f d\nu = \lim_{t\to \infty} \frac 1t \int_0^t f(\x u_+(s)) ds= \lim_{n\to \infty} \frac 1{t_n} \int_0^{t_n} f(\x u_+(s)) ds
 = \int_\X f d\omega_\X.
$$
Since this equality is true for every $f\in C_c(\X)$, we have $\nu=\omega_\X$.
\end{proof}

\vspace{1.5mm}

Now suppose that for some $\x\in \X-\P_\X$ and some $f\in C_c(\X)$, we do not have
$$
\int_\X f d\omega_\X = \lim_{t\to \infty} \frac 1t \int_0^t f(\x u_+(s)) ds.
$$

We can consider the set of probability measures $\nu_t$ defined by
$$
\int_\X f d\nu_t =  \frac 1t \int_0^t f(xu_+(s)) ds.
$$
On a non-compact space, the Riesz representation theorem says that set of Borel measures is the dual of the Banach space
$C_0(X)$ of continuous functions vanishing at $\infty$ with the sup norm.  The collection of probability measures
$\nu_t$ is a subset of the unit ball, and the unit ball  is compact for the weak topology. So if $\lim_{t\to\infty} \nu_t \ne \omega_\X$, there exists
a measure $\nu \ne \omega_\X$ and a sequence $t_i\to \infty$ such that
$$
\lim_{i\to \infty} \nu_{t_i} =\nu
$$
in the weak topology.

Clearly $\nu$ is invariant under the horocycle flow and ergodic. So it might seem that $\omega_\X\ne \nu$ contradicts
proposition \ref{classmeasures}. There is a difficulty with this argument when $\X$ is not compact.  In that case the
probability measures do not form a closed subset of the unit ball of $C_0(\X)^*$; consider for instance the measures $\delta(x-n)$
on $\R$; as $n\to \infty$ they tend to $0$ in the weak topology. Technically, the problem is that we can't evaluate measures on the
continuous function $1$, since this function doesn't vanish at infinity.

We need to show that $\nu$ is a probability measure.  This follows from proposition \ref{takelongtogetthere} below.  For $\rho \le 2$,
 let $\X^\rho \subset \X$ be the compact part of $\X$ in which all periodic horocycles have length $\ge \rho$.  So $\X_c=\X^2$\\

\begin{proposition} \label{takelongtogetthere} .  For any $\epsilon>0$, there exists $\rho>0$ such that for all $\x\in \X-\P_\X$ we have
$$
\lim_{t\to \infty} \frac 1t \int_0^t \one_{\X^\rho}(\x u_+(s)) ds >1-\epsilon.
$$
\end{proposition}

\begin{figure}[!ht]
\begin{centering}
\includegraphics[scale = 0.5] {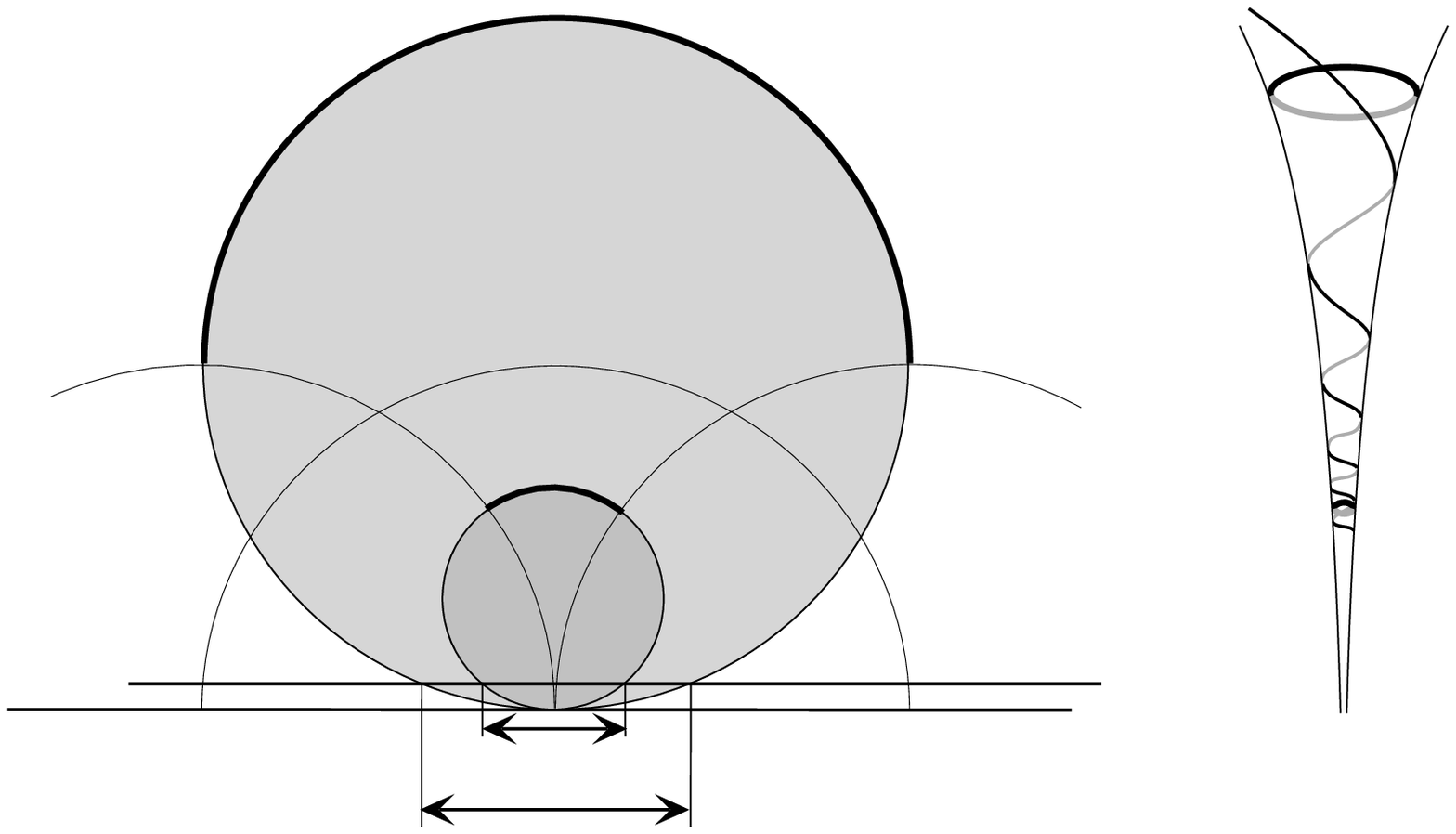}
\put (-162,-10){\scriptsize{$\sqrt{2\epsilon-\epsilon^2}$}}
\put (-162,4){\scriptsize{$\sqrt{\rho\epsilon-\epsilon^2}$}}
\put (-147, 120){\scriptsize 2}
\put (-147, 48){\scriptsize{$\rho$}}
\put (-147, 120){\scriptsize 2}
\put (-147, 120){\scriptsize 2}
\put (-55,19){\scriptsize {$\epsilon$}}
\put (-10, 49){\scriptsize{$\rho$}}
\caption{\label{enteringandleavinghorocycle} In $H$ a neighborhhod of a cusp bounded by a horocycle corresponds to a
disc tangent to the $x$-axis. In the figure on the left, we have represented the cusp by $\langle \gamma\rangle \backslash H$;
without loss of generality we may set $\gamma(z) = z/(1+z)$. Then the disc of radius $1$ centered at $i$ corresponds to the
neighborhood of the cusp bounded by the horocycle of length $2$, and the disc of radius $\rho/2$ centered at $i\rho/2$ corresponds
to the neighborhood bounded by a horocycle of length $\rho$. A horocycle that enters this neighborhood but does not go to the cusp
can be, without loss of generality, represented by a line of equation $y=\epsilon$; it goes deeper and deeper into the cusp as
$\epsilon\to 0$.  The ratio of times spent in $N^\rho$ to the time spent in $N^2-N^\rho$ does not become large as the horocycle goes
deeper in the cusp, but tends to a ratio depending only on $\rho$, which tends to $0$ as $\rho$ tends to $0$.  As horocycles go
deeper and deeper in the cusp, they spiral more and more tightly in $N^2-N^\rho$ and still spend approximately the same fraction of time in $N^2-N^\rho$ as in $N^\rho$.}
\end{centering}
\end{figure}

\begin{proof} If $c$ is a cusp of $X$, let  $N_c^\rho$ be the neighborhood of $c$ bounded by the horocycle of length $\rho$.
Recall from our discussion of the geometry of hyperbolic surfaces, that $N^2_c$ is isometric to a standard object: the part of $(2\Z)\backslash \H$ where $y>1$. Set $\gamma$ to be the Moebius transformation $\gamma(z)=z/(z+1)$; the standard neighborhood is then isometric to the part of
$\langle \gamma \rangle\backslash \H$ where $x^2+(y-1)^2\le 1$. Moreover, any horocycle that doesn't tend to the cusp is equivalent
by a change of variable commuting with $\gamma$ to a horizontal line. Of course lengths on such a horizontal line $y=\epsilon$
depend on $\epsilon$, but ratios of lengths are the same as ratios of euclidean lengths.

A careful look at figure
\ref{enteringandleavinghorocycle} shows that if a horocycle starts in
$\X_c$, goes deep in the cusp, and comes out again, then the ratio of time spent in $N^\rho$ to time spent in $N^2-N^\rho$ is
\begin{equation}
\frac{\sqrt{\rho\epsilon-\epsilon^2}}{\sqrt{2\epsilon-\epsilon^2}-\sqrt{\rho\epsilon-\epsilon^2}}= \frac{\sqrt \rho}{\sqrt 2-\sqrt
\rho} + O(\epsilon).
\end{equation}
Any non-periodic horocycle will eventually enter $\X_c$; by taking $\rho$ sufficiently small, we can assure that afterwards
it will spend a proportion of its time $<\epsilon$ outside of $X^\rho$.  Proposition \ref{takelongtogetthere} follows.
\end{proof}

\vspace{2mm}

Consider the measures
$$
\nu_{\x,T} = (f \mapsto \frac 1T \int_0^T f(\x u_+(t) dt).
$$

\begin{proposition} The accumulation set of $\{\nu_{\x,T}, T>0\}$ consists entirely of probability measures.
\end{proposition}

\begin{proof} Every accumulation point $\mu$ of the $\nu_{\x,T}$ in $C_0(\X)^*$ is a measure, and the only thing to show is that
$\mu(\X)=1$.  Clearly $\mu(\X)\le 1$, since for any $f\in C_0(X)$ and any $\x,T$ we have
$$
\frac 1T \int_0^T f(\x u_+(t)) dt\le \|f\|_\infty.
$$

To see that $\mu(\X)\ge 1$, take $\epsilon>0$ and  $\rho$ as in proposition \ref{takelongtogetthere}. We can then find a function
$f\in C_0(\X)$ which coincides with $\one_{\X^\rho}$ on $\X^\rho$ and satisfies $0\le f \le 1$ everywhere.
Then
\begin{align}
\mu(\X)&= \sup_{g\in C_0(\X)} \frac{\int_\X |g d\mu|}{\|g\|_\infty}\\
&\ge \int_\X f d\mu \ge \liminf_{T\to \infty} \frac 1T \int_0^T f(\x u_+(t)) dt \\
&\ge \liminf_{T\to \infty} \frac 1T \int_0^T \one_{\X^\rho}(\x u_+(t)) dt > 1-\epsilon.
\end{align}
\end{proof}

\vspace{2mm}

There is one last thing to check.

\begin{proposition} A measure $\mu$ in the limit set of $\{\nu_{\x,T}, T>0\}$ with $\x\notin \P_\X$ satisfies $\mu(\P_\X)=0$.
\end{proposition}

\begin{proof} Suppose $\mu(\P_\X) >0$, set $\epsilon= \mu(\P_\X)/3$ and use proposition \ref{takelongtogetthere} to find a
corresponding $\rho$. Find a compact subset $\Q\subset \P_\X$ with $\mu(\Q)>\frac 23 \mu(\P_\X)$, and find a time $T$ such that
$$
\Q g(T) \cap \X^\rho=\emptyset.
$$
This is possible because $\P_\X$ consists of points in the stable manifolds of the cusps, so each point can be moved off $\X^\rho$,
and since $\Q$ is compact it will leave $\X^\rho$ under the geodesic flow at some time $T$.

Let $\U$ be a neighborhood of $\Q$ such that $\U g(T)\cap \X^\rho=\emptyset$. For this neighborhood $\U$ of $\Q$, as for any
neighborhood, there exists a sequence of times
$T_n
\to
\infty$ such that
$$
\frac{\lambda\{t\in [0,T_n]\ |\ \x u_+(t)\in \U\} }{T_n} > \frac 12 \mu (\Q),
$$
where $\lambda$ is linear measure. Then the horocycle $t\mapsto \x g(T) u_+(t)$ must spend the same proportion of its time in $\U
g(T)$, hence outside
$\X^\rho$.  But every non-periodic horocycle spends at least a proportion $1-\mu(\P_X)/3$ in $\X-\X^\rho$, and this is a contradiction.
\end{proof}

\section{Horocycle flow on the modular surface}  \indent Let $\Gamma$ be the 2-congruence subgroup of $\SL_2(\R)$, so that $\X= \Gamma\backslash\SL_2\R$ is the unit tangent bundle over $X=\Gamma\backslash H$, which is the 3-times punctured sphere.  Denote by $\pi_\X : \H \to \X$ and $\pi_X :\X \to X$ respectively the projections from  $\H \cong \PSL(2,\R)$ onto $\X \cong \Gamma \backslash \PSL(2,\R)$ and from $\X$ onto $X$.

\begin{lemma} \label{sizeofcusps} The hyperbolic surface $X$ has area $2\pi$, and the subset $X- X^\rho$ has area $3\rho$. \end{lemma}

We leave the proof of this lemma to the reader.

It follows from lemma \ref{sizeofcusps} that for every $\x_0 \notin \P_\X$ there exists for every sequence $\rho_n \to 0$, for  every $\epsilon>0$ and for
every $n$ sufficiently large, a time $$ T_n< (1+\epsilon) \left(\frac{2\pi}{3\rho_n}\right) $$ such that $\pi_X(\x_0u_+(T_n))\in X- X^{\rho_n}$.

To use this result, we need to understand the region in $H$ corresponding to $X^\rho$.

\begin{lemma}\label{sizeofhodiscs} The inverse image in $H$ of $X-X^\rho$ is the union of the horodisc $\im z>2/\rho$, and the union, for all rational
numbers $p/q$ of the discs of radius $\rho/(4q^2)$ tangent to the real axis at $p/q$. \end{lemma}

Choose $\alpha \in \R-\Q$, and consider the  horocycle in $\X$ which is the image of the horocycle $\z_0 u_+(t)$ in $\H$, where $\z_0=\alpha+2i$, i.e., the image of the horocycle represented by the circle of radius 1 tangent to the real axis at the irrational number $\alpha$.  A straightforward computation shows that
$$
\z_0 u_+(T)=\left(\alpha-\frac{2T}{T^2+1}\right)+i\frac 2{T^2+1}.
$$

Set $\rho_n=1/n$. This horocycle is not periodic, so it must enter $\X-\X^{\rho_n}$ at a sequence of times $T_n<(1+\epsilon)\left(\frac{2\pi n}3\right)$. Interpreting the cusps as rational numbers, this means that there  exists an infinite sequence of rational numbers $p_n/q_n$ and times $T_n<(1+\epsilon)\frac{2\pi n}{3}$ such that 
$$
\left|\alpha-\frac{p_n}{q_n}\right|= \frac {2T_n}{T_n^2+1}\le \frac {T_n}{2nq_n^2}<(1+\epsilon)\frac{\pi n}{3nq_n^2} = (1+\epsilon)\frac{\pi}{3q_n^2}. 
$$

This is of course nothing to boast about.  It has been known for over 100 years that for every irrational number $\alpha$, there exist infinitely many
coprime numbers  numbers $p_n,\ q_n$ such that $$ \left|\alpha-\frac {p_n}{q_n}\right| < \frac 1{\sqrt 5 q_n^2} , $$ and that $1/\sqrt 5$ is the smallest
number for which this is true \cite{Kinchin64}.  Our analysis only gives the constant $\pi/3$, too large by a factor of more than  2.

One reason to take an interest in this result despite its weakness is that Ratner's theorem has many generalizations to situations where methods leading to
the  sharp results about diophantine approximations of irrational numbers are not available.  In all settings, Ratner's theorem has ``diophantine''
consequences.

Clearly we cannot do better than improve the constant for all horocycles. But we can use the theory of diophantine approximations to improve the results
above for almost every horocycle. In particular we can apply the following theorem.

\begin{theorem}\label{almostalldioph} \textrm{\cite{Kinchin64}} If $ g(x): \R_+^* \to \R$ is a function such that $g(x)/x$ is increasing, then for almost every $\alpha$, there exist infinitely many coprime integers $p,q$ such that $ |\alpha-\frac  pq| <\frac 1{qg(q)}$ if and only if the series $$ \sum_{n=1}^\infty \frac 1{g(n)} $$ diverges. \end{theorem}

Let us see what this says about horocycles; we will specialize to the case where $g(n)=n\log(n)$.  For almost every $\x_0$, the horocycle $\x_0 U_+$ lifts to a horocycle $\z_0 U_+$ in $\H$ tangent to the $x$-axis at an irrational number $\alpha$ belonging to the set of full measure from theorem  \ref{almostalldioph}. Changing the time parameterization by a constant, we may assume that $\z_0= (\alpha +2iR,-2iR)$ since this is in any case the worst point on the horocycle $\z_0 U_+$.

\begin{figure}[!ht] \begin{centering} \includegraphics[scale = 0.5] {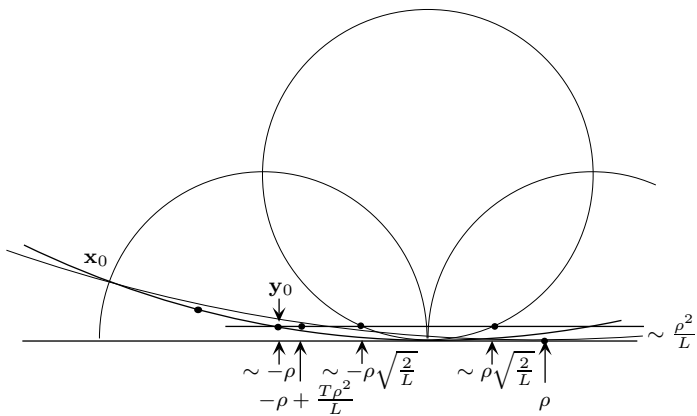} \put (-216,45){\scriptsize{$\x_0$}} \put (-146,34){\scriptsize{$\y_0$}}
\put (-44,-9){\scriptsize{$\rho$}} \put (-150,-9){\scriptsize{$-\rho+\frac {T\rho^2}L$}} \put (-156,2){\scriptsize{$\sim -\rho$}} \put
(-126,2){\scriptsize{$\sim -\rho\sqrt{\frac 2L}$}} \put (-75,2){\scriptsize{$\sim \rho\sqrt{\frac 2L}$}} \put (-3,16){\scriptsize{$\sim \frac{\rho^2}{L}$}}
\caption{\label{Breuillardhorocycle} We can lift the horocycle $\x_0 U_+$ to $\H$; without loss of generality we may assume that the cusp $c$ is at $0$,
and that the stabilizer of the cusp is generated by $z \mapsto z/(z+1)$. In that case, the horocycle of length $2$ lifts to the circle of radius $1$ centered at
$i$, and the horocycle of length $L$ lifts to the circle of radius $L/2$ centered at $Li/2$. We may take $x_0=\pi_X(\x_0)$ to be anywhere on this horocycle; it will be
convenient to place it at 
$\frac{-2L^2+4Li}{L^2+4}$.  In that case, one fundamental domain on the horocycle goes from $\frac{-2L^2+4Li}{L^2+4}$
to $\frac{2L^2+4Li}{L^2+4}$. Our modified horocycle will join $x_0$ to the point $\rho>0$ on the real axis. It is much easier to estimate lengths on this
horocycle if we send $\rho$ to infinity by a parabolic transformation that fixes $0$, and hence all the horocycles tangent to the real axis at $0$.  If we
perform this parabolic transformation, the point $x_0$ moves to a point $y_0$ on its horocycle which is approximately $-\rho +i \rho^2/L$, and the
horocycle is a horizontal line, approximately the line $y=\rho^2/L$.} \end{centering} \end{figure}

Let $p_n/q_n$ be one of the good approximations to $\alpha$ guaranteed by theorem \ref{almostalldioph}.  Let $\rho_n$ be the radius of the negative
horocycle surrounding the cusp corresponding to $p_n/q_n$ when the point $\z_0 u_+(T_n)$ is on the vertical line $x=p_n/q_n$. Further let us write
$$
\z_0 u_+(T_n)= \xi_n+i\eta_n= 2R\left(\frac {T_n}{T_n^2+1} +\frac i{T_n^2+1}\right).
$$

Then we have $T_n= \frac {\xi_n}{\eta_n}$ and $\eta_n= \frac {\rho_n}{2 q_n^2}$ and by lemma \ref{sizeofhodisks} it follows that
$$
T_n= \frac {\xi_n}{\eta_n} \le \frac{1/(q_n^2 \log q_n)}{\rho_n / (2 q_n^2)}= \frac {2}{\rho_n \log q_n}.
$$

Moreoever $\eta_n=R-\sqrt(R^2-\xi_n^2) \sim \frac{2 \xi_n^2}{R}$ from which we can derive that $q_n \sim \frac{1}{2 \log(\frac{1}{\rho_n})}$.\\

We have proved the following:
  
\begin{theorem} \label{betterthaneqd} On the modular surface, for every $\epsilon >0$ and for almost every horocycle $\x_o U_+$, there exists a sequence $\rho_n \rightarrow 0$ and times $T_n< (1+\epsilon)\frac{1}{\rho_n \log \frac{1}{\rho_n}}$ such that $\x_o u_+(T_n) \in \X-\X^{\rho_n}$.
\end{theorem}

The theorem means that almost every nonperiodic horocycle enters $\X-\X^{\rho_n}$ much earlier than is implied for every nonperiodic horocycle by equidistribution.

This leads to a surprising result due to Breuillard \cite{Breuillard05}: although non-periodic horocycles are equidistributed, \textit{any} uncentered random walk on the set of non-periodic  horocycles almost surely is not.

\begin{theorem}\label{Breuillard} Let $\mu$ be a probability measure on $\R$ with finite expectation and variance:
$$
0 \neq a= \int_{-\infty}^\infty  t \mu(dt) <\infty \quad \text{and} \quad b^2= \int_{-\infty}^\infty  (t-a)^2 \mu(dt) <\infty.
$$
Denote by $\mu^{*m}$ the $m$th convolution of $\mu$ with itself.  If $b>0$, there exists a function $f\in C_c(\X)$ with $\int_\X f(\x)  \omega_X(d\x)=1$ such that for almost every $\x_0 \in \X$ we have $$ \liminf_{m \to \infty} \int_{-\infty}^\infty f(\x_0u_+(t)) \mu^{\star m}(dt) = 0. $$ \end{theorem}

\begin{proof} Let $\alpha \in \R$ be an arbitrary element of the full Lebesgue measure set guaranteed by theorem \ref{almostalldioph}.  Suppose that the horocycle $\z_0 U_+$ through $\z_0 =(\alpha + 2iR,-2iR)$ in $\H$ projects to a horocycle in $\X$ containing $\x_0$.  Now choose an approximating sequence of coprime $(p_n, q_n) \in \Z^2$ to $\alpha$ as guaranteed by theorem \ref{almostalldioph} above, and let $T_n$ be the associated sequence of times.   The measure $\mu^{\star m}$ is approximately the Gaussian of mean $ma$ and standard deviation $\sqrt m b$. Let us choose an $m$ such that:

\begin{enumerate}
\item $ma=T_n$ for one of the $T_n$ given in theorem \ref{betterthaneqd}
\item the standard deviation of $\sigma(\mu^{\star m}) \sim b\sqrt m$ is much smaller than
$1/\sqrt{\rho_n}$
\end{enumerate}

 This first condition can obviously be satisfied, and the second is also straightforward since
$$
b\sqrt m \sim b \sqrt {\frac{T_n}a} \le \frac b{\sqrt a} \sqrt {\frac 2{\rho_n \log q_n}}
$$
and $\sigma(\mu^{\star m})$ will be much smaller than $1/\sqrt{\rho_n}$ as soon as $q_n$ is large enough.

Recall that it takes time of the order $1/\sqrt{\rho}$ for a horocycle to get from $X-X^\rho$ to $X^2$.  Thus for the $m$ found above, there are many, say $c(m)$,
standard deviations of $\mu^{\star m}$ around the mean $am$ where $\x_0 u_+([am-c(m),am+c(m)]) \in \X-\X^2$.  It follows that if $f\in C_c(\X)$ satisfies $\int_\X f(\x)
\omega_X(d\x)=1$ but $f$ has its support in $\X^2$, we have
$$
\liminf_{m \to \infty} \int_{-\infty}^\infty f(\x_0 u_+(t)) \mu^{\star m}(dt) =0.
$$
This proves that the random walk is not equidistributed.
\end{proof}

\nocite{Furstenberg73}
\nocite{Ratner91}
\nocite{Ratner91_n2}
\nocite{Ratner91_n3}
\nocite{Ratner92}

\bibliographystyle{alpha}
\bibliography{RatnerBiblio0}

\end{document}